\documentclass[12pt,reqno,a4paper]{amsart}
\usepackage{hyperref}
\usepackage{enumerate}
\usepackage[left=0.6in, right=0.6in, top= 0.9in, bottom=1.55in]{geometry}
\usepackage{amsmath,amssymb,amsthm,amsfonts}

\usepackage{tikz}
\usetikzlibrary{positioning}

\usepackage{verbatim}
\usetikzlibrary{arrows,shapes}
\usetikzlibrary{positioning, 
	quotes}

\usepackage{graphicx}
\usepackage{tikz}
\usetikzlibrary{graphs, graphs.standard}

\tikzset{
	modal/.style={>=stealth’,shorten >=1pt,shorten <=1pt,auto,node distance=1.5cm,
		semithick},
	world/.style={circle, draw,minimum size=.1cm,fill=gray!15},
	point/.style={circle,draw,inner sep=0.3mm,fill=black},
	circ/.style={circle,draw,inner sep=0.1mm,fill=white},
	reflexive above/.style={->,loop,looseness=7,in=120,out=60},
	reflexive below/.style={->,loop,looseness=7,in=240,out=300},
	reflexive left/.style={->,loop,looseness=7,in=150,out=210},
	reflexive right/.style={->,loop,looseness=7,in=30,out=330}
}
\pgfdeclarelayer{background}
\pgfsetlayers{background,main}

\tikzstyle{vertex}=[circle,fill=black!25,minimum size=14pt,inner sep=0pt]
\tikzstyle{selected vertex} = [vertex, fill=red!24]
\tikzstyle{edge} = [draw,thick,-]
\tikzstyle{weight} = [font=\small]
\tikzstyle{selected edge} = [draw,line width=5pt,-,red!50]
\tikzstyle{ignored edge} = [draw,line width=5pt,-,black!20]

\usetikzlibrary{shapes}
\usetikzlibrary{plotmarks}
\usetikzlibrary{arrows}
\usetikzlibrary{positioning}
\theoremstyle{definition}
\newtheorem{defn}{Definition}[section]
\newtheorem{fact}[defn]{Fact}

\newtheorem{thm}[defn]{Theorem}

\newtheorem{corr}[defn]{Corollary}
\newtheorem{lem}[defn]{Lemma}

\newtheorem{remark}[defn]{Remark}

\newtheorem{claim}[defn]{Claim}
\title[Distinguishing chromatic number of middle and subdivision graphs]{Distinguishing chromatic number of middle and subdivision graphs}

\author{Amitayu Banerjee}
\address{Alfr\'ed R\'enyi Institute of Mathematics, Reáltanoda utca 13-15, 1053, Budapest, Hungary}
\email[Corresponding author]{banerjee.amitayu@gmail.com}

\author{Alexa Gopaulsingh}
\address{E\"otv\"os Lor\'and University, Department of Logic, M\'{u}zeum krt. 4, 1088, Budapest, Hungary}
\email{alexa279e@gmail.com}

\author{Zal\'{a}n Moln\'{a}r}
\address{E\"otv\"os Lor\'and University, Department of Logic, M\'{u}zeum krt. 4, 1088, Budapest, Hungary}
\email{mozaag@gmail.com}

\date{}

\makeatletter
\@namedef{subjclassname@2020}{\textup{2020} Mathematics Subject Classification}
\makeatother
\begin{document}
	\begin{abstract}
		Let $G$ be a simple finite connected graph of order $n\geq 3$ with maximum degree $\Delta(G)$. In 2016, Kalinowski, Pil\'{s}niak, and Wo\'{z}niak introduced the total distinguishing number $D''(G)$ of $G$. 
		We prove the following and show that the upper bound mentioned in (3) is sharp:
		\begin{enumerate}
			\item The distinguishing chromatic number $\chi_{D}(M(G))$ of the middle graph $M(G)$ of the graph $G$ is $\Delta(G)+1$ except for four small graphs $C_{4}, K_{4}, C_{6}$, and $K_{3,3}$, and $\Delta(G)+2$ otherwise.

			\item Inspired by a recent result of Mirafzal, we show that the distinguishing number $D(S(G))$ of the subdivision graph $S(G)$ of $G$ is $D''(G)$.
			
			Consequently, $D(S(G))$ is at most $\lceil \sqrt{\Delta(G)}\rceil$.
			\item Let $G\not\cong C_{n}$, where $C_{n}$ is the cycle graph of order $n$. If the distinguishing number $D(G)$ of $G$ is at least 3, then the distinguishing chromatic number $\chi_{D}(S(G))$ of $S(G)$ is at most $D(G)$, and if $D(G)$ is at most $2$, then $\chi_D(S(G))= D(G)+1$.
			\item If $D(G)\neq 1$ 
			and $\chi_D(G)=2$, then the automorphism group of $G$ consists of $2$ elements.
		\end{enumerate}
	\end{abstract}
	
	\maketitle
	\section{Introduction and definitions}
	
	Let $G = (V(G), E(G))$ be a connected graph with the vertex set $V(G)$ and the edge set $E(G)$, and let 
	$Aut(G)$ be its automorphism group. 
	The {\em line graph} $L(G)$ of $G$ is the graph with vertex set $E(G)$, where vertices $x$ and $y$ are
	adjacent in $L(G)$ if and only if the corresponding edges $x$ and $y$ share a common vertex in $G$. 
	The {\em subdivision graph} $S(G)$ of $G$ is obtained from $G$ by putting a new vertex in
	the middle of every edge of $G$. In \cite{HY1976}, Hamada and Yoshimura introduced the middle graph of a graph. 
	The {\em middle graph} $M(G)$ of $G$ is the graph whose vertex set is $V(G) \cup E(G)$,
	and two vertices $x, y$ in the vertex set of $M(G)$ are adjacent in $M(G)$ in case one of the following holds:
	\begin{enumerate}
		\item $x, y \in E(G)$ and $x, y$ are adjacent in $G$;
		\item $x \in V (G), y \in E(G)$, and $x, y$ are incident in $G$.
	\end{enumerate}
	
	Albertson and Collins \cite{AC1996} studied the distinguishing number of a graph and the distinguishing chromatic number of a graph was introduced by Collins and Trenk \cite{CT2006}. 
	Let $f$ be an assignment of colors to either vertices or edges of $G$. An automorphism $\psi$ of $G$ {\em preserves} $f$ if each vertex (edge) of $G$ is mapped to a vertex (edge) of the same color.
	Let $id_{G}$ denote the trivial automorphism of $G$.
	A vertex coloring $h:V(G)\rightarrow\{1,...,r\}$ is {\em $r$-distinguishing} if $h$ is preserved only by $id_{G}$. 
	The coloring $h$ is {\em $r$-proper distinguishing} if $h$ is $r$-distinguishing and a proper coloring. The {\em distinguishing number} of $G$, denoted by $D(G)$, is the least integer $r$ such that $G$ has an $r$-distinguishing vertex coloring. The {\em distinguishing chromatic number} of $G$, denoted by $\chi_{D}(G)$, is the least integer $r$ such that $G$ has an $r$-proper distinguishing vertex coloring. Let $\chi(G)$ denote the chromatic number of $G$. The open neighborhood of $v \in V (G)$ is the set $N_{G}(v) =\{u \in V (G) : \{u,v\} \in E(G)\}$.

	\subsection{Motivation and results} Let $C_{n}$ be a cycle graph of $n$ vertices, $K_{n}$ be a complete graph of $n$ vertices, and $K_{n,m}$ be a complete bipartite graph with bipartitions of size $n$ and $m$. 
	\begin{itemize}
		\item 
		In 1998, Nihei \cite{Nih1998} established that $\chi(M(G)) = \Delta(G) + 1$ for any graph $G$. 
		In Section 2, we apply a result of Hamada–Yoshimura \cite{HY1976} and two recent results of Alikhani–Soltani \cite{AS2020}, and Kalinowski–Pil\'{s}niak \cite{KP2015} to determine $\chi_{D}(M(G))$ of a connected graph $G$.
		In particular, we prove that if $G$ is a connected graph of order $n\geq 3$, then $\chi_{D}(M(G))=\Delta(G)+1$ if $G\not\in\{C_{4}, K_{4}, C_{6}, K_{3,3}\}$ and $\chi_{D}(M(G))=\Delta(G)+2$ otherwise.
		
		\item In 2016, Kalinowski, Pil\'{s}niak, and Wo\'{z}niak \cite{KPW2016} introduced the total distinguishing number $D''(G)$ of a graph $G$ as the least integer $d$ such that $G$ has a total coloring (not necessarily proper) 
		with $d$ colors that is preserved only by $id_{G}$. In the same paper, Kalinowski, Pil\'{s}niak, and Wo\'{z}niak  proved that $D''(G)\leq \lceil \Delta(G)\rceil$ if $G$ is a connected graph of order $n \geq 3$. Recently, Mirafzal \cite{Mir2024} determined
		the automorphism group of $S(G)$ and proved that if $G \not\cong C_{n}$ for any natural number $n$, then $Aut(G) \cong Aut(S(G))$. In Section 3, inspired by the techniques in \cite{Mir2024}, the first author proves that if $G$ is a connected graph of order $n \geq 3$,  then $D(S(G))=D''(G)$. Consequently, $D''(G)\leq \lceil \Delta(G)\rceil$.

		\item Let $G$ be a connected graph such that $G\not\cong C_{n}$ for any natural number $n$. In Section 4, we apply the techniques in \cite{Mir2024} to show that if $D(G)\geq 3$, then $\chi_{D}(S(G))\leq D(G)$, and if $D(G)=2$, then $\chi_D(S(G))= 3$. We also observe that if $D(H)\neq 1$ and $\chi_D(H) = 2$, then $\vert Aut (H)\vert = 2$ for any connected graph $H$.
	\end{itemize}
	
	\section{Distinguishing chromatic number of Middle graphs}
	
	We have the following definition due to Hamada and Yoshimura \cite{HY1976}.
	
	\begin{defn}\label{Definition 2.1}
		Let $G=(V(G), E(G))$ be a graph where $V(G) = \{v_{1},..., v_{p}\}$ for some natural number $p$. To $G$, we add
		$p$ new vertices $\{u_{1},..., u_{p}\}$ and $p$ new edges $\{u_{i}, v_{i}\}$. The new graph is the {\em endline graph} of $G$ denoted by $G^{+}$. The edges $\{u_{i}, v_{i}\}$ are the {\em endline edges} of $G^{+}$. 
	\end{defn}
	
	Kalinowski and Pil\'{s}niak \cite{KP2015} introduced the distinguishing index $D'(G)$ and the distinguishing chromatic index $\chi_{D}'(G)$ of a graph $G$.
	
	\begin{defn}\label{Definition 2.2}
		An edge coloring $h:E(G)\rightarrow\{1,...,r\}$ of a graph $G=(V(G), E(G))$ is {\em $r$-distinguishing} if $h$ is preserved only by $id_{G}$. 
		The {\em distinguishing index} of $G$, denoted by $D'(G)$, is the least integer $r$ such that $G$ has an $r$-distinguishing edge coloring and the {\em distinguishing chromatic index} of $G$, denoted by $\chi_{D}'(G)$, is the least integer $r$ such that $G$ has an $r$-proper distinguishing edge coloring.
	\end{defn}
	
	\begin{fact}\label{Fact 2.3}
		{\em The following holds:
			\begin{enumerate}
				\item (Hamada--Yoshimura; \cite{HY1976}) Let $G$ be any graph. Then $M(G) \cong L(G^{+})$.
				
				\item (Alikhani--Soltani; \cite[Theorem 2.10]{AS2020}) If $G$ is a connected graph of order $n\geq 3$, that is not $Q$ and $L(Q)$ (see Figure \ref{Figure 1}). Then $\chi_{D}(L(G))=\chi_{D}'(G)$.
				
				\item (Kalinowski--Pil\'{s}niak; \cite[Theorem 16]{KP2015}) If $G$ is a connected graph of order $n \geq 3$, then $\chi_{D}'(G) \leq \Delta(G) + 1$, except for four graphs of small order $C_{4}, K_{4}, C_{6}, K_{3,3}$. For these exceptional cases, $\chi_{D}'(G) = \Delta(G) + 2$.
			\end{enumerate}    
		}
	\end{fact}
	
	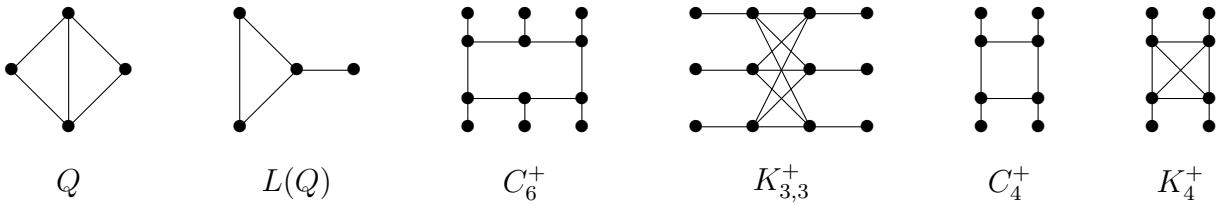
\begin{figure}[!ht]
		\centering
		\begin{minipage}{\textwidth}
			\centering
			\begin{tikzpicture}[scale=0.75]
				\draw[black,] (-10,-1) -- (-9,-2);
				\draw[black,] (-10,-1) -- (-11,-2);
				\draw[black,] (-10,-3) -- (-11,-2);
				\draw[black,] (-10,-3) -- (-9,-2);
				\draw[black,] (-10,-3) -- (-10,-1);
				
				\draw (-10,-1) node {$\bullet$};
				\draw (-9,-2) node {$\bullet$};
				\draw (-11,-2) node {$\bullet$};
				\draw (-10,-3) node {$\bullet$};
				
				\draw (-10,-4) node {$Q$};
				\draw[black,] (-7,-1) -- (-6,-2);
				\draw[black,] (-7,-3) -- (-6,-2);
				\draw[black,] (-7,-3) -- (-7,-1);
				\draw[black,] (-6,-2) -- (-5,-2);
				
				\draw (-7,-1) node {$\bullet$};
				\draw (-6,-2) node {$\bullet$};
				\draw (-5,-2) node {$\bullet$};
				\draw (-7,-3) node {$\bullet$};
				
				\draw (-6,-4) node {$L(Q)$};
				\draw[black,] (-3,-1.5) -- (-1,-1.5);
				\draw[black,] (-3,-2.5) -- (-1,-2.5);
				\draw[black,] (-3,-2.5) -- (-3,-1.5);
				\draw[black,] (-1,-2.5) -- (-1,-1.5);
				\draw[black,] (-3,-1.5) -- (-3,-1);
				\draw[black,] (-3,-2.5) -- (-3,-3);
				\draw[black,] (-2,-1.5) -- (-2,-1);
				\draw[black,] (-2,-2.5) -- (-2,-3);
				\draw[black,] (-1,-1.5) -- (-1,-1);
				\draw[black,] (-1,-2.5) -- (-1,-3);
				\draw (-3,-1.5) node {$\bullet$};
				\draw (-3,-2.5) node {$\bullet$};
				\draw (-1,-1.5) node {$\bullet$};
				\draw (-1,-2.5) node {$\bullet$};
				\draw (-2,-1.5) node {$\bullet$};
				\draw (-2,-2.5) node {$\bullet$};
				\draw (-3,-1) node {$\bullet$};
				\draw (-3,-3) node {$\bullet$};
				\draw (-1,-1) node {$\bullet$};
				\draw (-1,-3) node {$\bullet$};
				\draw (-2,-1) node {$\bullet$};
				\draw (-2,-3) node {$\bullet$};
				
				\draw (-2,-4) node {$C_{6}^{+}$};
				\draw (1,-1) node {$\bullet$};
				\draw (1,-2) node {$\bullet$};
				\draw (1,-3) node {$\bullet$};
				\draw (2,-1) node {$\bullet$};
				\draw (2,-2) node {$\bullet$};
				\draw (2,-3) node {$\bullet$};
				\draw (3,-1) node {$\bullet$};
				\draw (3,-2) node {$\bullet$};
				\draw (3,-3) node {$\bullet$};
				\draw (4,-1) node {$\bullet$};
				\draw (4,-2) node {$\bullet$};
				\draw (4,-3) node {$\bullet$};
				
				\draw[black,] (2,-1) -- (3,-1);
				\draw[black,] (2,-1) -- (3,-2);
				\draw[black,] (2,-1) -- (3,-3);
				\draw[black,] (2,-2) -- (3,-1);
				\draw[black,] (2,-2) -- (3,-2);
				\draw[black,] (2,-2) -- (3,-3);
				\draw[black,] (2,-3) -- (3,-1);
				\draw[black,] (2,-3) -- (3,-2);
				\draw[black,] (2,-3) -- (3,-3);
				
				\draw[black,] (1,-1) -- (2,-1);
				\draw[black,] (1,-2) -- (2,-2);
				\draw[black,] (1,-3) -- (2,-3);
				\draw[black,] (3,-1) -- (4,-1);
				\draw[black,] (3,-2) -- (4,-2);
				\draw[black,] (3,-3) -- (4,-3);
				
				\draw (2.5,-4) node {$K_{3,3}^{+}$};
				\draw (6,-1) node {$\bullet$};
				\draw (7,-1) node {$\bullet$};
				\draw (6,-1.5) node {$\bullet$};
				\draw (7,-1.5) node {$\bullet$};
				\draw (6,-2.5) node {$\bullet$};
				\draw (7,-2.5) node {$\bullet$};
				\draw (6,-3) node {$\bullet$};
				\draw (7,-3) node {$\bullet$};
				\draw[black,] (6,-1.5) -- (7,-1.5);
				\draw[black,] (6,-1.5) -- (6,-2.5);
				\draw[black,] (6,-2.5) -- (7,-2.5);
				\draw[black,] (7,-1.5) -- (7,-2.5);
				\draw[black,] (6,-1) -- (6,-1.5);
				\draw[black,] (6,-2.5) -- (6,-3);
				\draw[black,] (7,-1) -- (7,-1.5);
				\draw[black,] (7,-2.5) -- (7,-3);
				
				\draw (6.5,-4) node {$C_{4}^{+}$};
				\draw (9,-1) node {$\bullet$};
				\draw (10,-1) node {$\bullet$};
				\draw (9,-1.5) node {$\bullet$};
				\draw (10,-1.5) node {$\bullet$};
				\draw (9,-2.5) node {$\bullet$};
				\draw (10,-2.5) node {$\bullet$};
				\draw (9,-3) node {$\bullet$};
				\draw (10,-3) node {$\bullet$};
				\draw[black,] (9,-1.5) -- (10,-1.5);
				\draw[black,] (9,-1.5) -- (9,-2.5);
				\draw[black,] (9,-2.5) -- (10,-2.5);
				\draw[black,] (10,-1.5) -- (10,-2.5);
				\draw[black,] (9,-1.5) -- (10,-2.5);
				\draw[black,] (10,-1.5) -- (9,-2.5);
				\draw[black,] (9,-1) -- (9,-1.5);
				\draw[black,] (9,-2.5) -- (9,-3);
				\draw[black,] (10,-1) -- (10,-1.5);
				\draw[black,] (10,-2.5) -- (10,-3);
				
				\draw (9.5,-4) node {$K_{4}^{+}$};
				
			\end{tikzpicture}
		\end{minipage}
		\caption{\em Graphs $Q, L(Q), C_{6}^{+}, K_{3,3}^{+}$, $C_{4}^{+},$ and $K_{4}^{+}$.}
		\label{Figure 1}
	\end{figure}
	
	\begin{lem}\label{Lemma 2.4}
		{\em If $\psi\in Aut(G^{+})$ such that $\psi(u)=u$ for all $u\in V(G)$, then $\psi=id_{G^{+}}$.}
	\end{lem}
	
	\begin{proof}
		Let $u_{p}\in V(G^{+})\backslash V(G)$. If $\psi(u_{p})=u_{q}$ for some $u_{q}\in  V(G^{+})\backslash (V(G)\cup \{u_{p}\})$, then $\psi(v_{p})=v_{q}$ since automorphisms  preserve adjacency, which is a contradiction to the given assumption. 
	\end{proof}
	
	\begin{lem}\label{Lemma 2.5}
		{\em If $G\in \{C_{4}, C_{6}, K_{4}, K_{3,3}\}$,
			then $\chi_{D}'(G^{+})=\Delta(G)+2$.}
	\end{lem}
	
	\begin{proof}
		Fix $G\in \{C_{4}, C_{6}, K_{4}, K_{3,3}\}$.
		By Fact \ref{Fact 2.3}(3), $\chi_{D}'(G)=\Delta(G)+2$. We use the property that all vertices of $G$ has same degree, i.e. $\Delta(G)$, to claim that $\chi_{D}'(G^{+})\geq\Delta(G)+2$. 
		For the sake of contradiction, assume that $f:E(G^{+})\rightarrow \{1,...,\Delta(G)+1\}$ is a proper distinguishing edge coloring of $G^{+}$. We claim that $g=f\restriction E(G)$ is a proper distinguishing edge coloring of $G$ which contradicts the fact that $\chi_{D}'(G)=\Delta(G)+2$. Clearly, $g$ is a proper coloring. It is enough to show that $g$ is a distinguishing coloring of $G$. Let $\alpha$ be an automorphism of $G$ preserving $g$. 
		Define the mapping $\pi_{\alpha}: V(G^{+}) \rightarrow V(G^{+})$ by the following rule:
		
		\begin{align*}
			\pi_{\alpha}(w) =
			\begin{cases}
				\alpha(v_{i}) & \text{if } w=v_{i}\in V(G)\\
				u_{j} & \text{if } w=u_{i}\in V(G^{+})\backslash V(G), v_{j}= \alpha(v_{i}).
			\end{cases}
		\end{align*}
		
		Then $\pi_{\alpha}\in Aut(G^{+})$, where $\pi_{\alpha}(V(G)) = V(G)$ and $\pi_{\alpha}(V(G^{+})\backslash V(G)) = V(G^{+})\backslash V(G)$. 
		
		\begin{claim}\label{Claim 2.6}
			{\em $\pi_{\alpha}$ is the trivial automorphism of $G^{+}$.}
		\end{claim}
		\begin{proof}
			If $e\in E(G^{+})$, then either $e\in E(G)$ or $e$ is an endline edge of $G^{+}$. We show that $f(\pi_{\alpha}(e)) = f(e)$ for all edges $e\in E(G^{+})$.
			Fix $\{x,y\}\in E(G)$.
			Since $g=f\restriction E(G)$ and $\alpha$ preserves $g$, we have $f(\pi_{\alpha}(\{x,y\}))= f(\{\pi_{\alpha}(x), \pi_{\alpha}(y)\}) = f(\{\alpha(x), \alpha(y)\}) = f(\alpha(\{x,y\}))= g(\alpha(\{x,y\})) = g(\{x,y\})=f(\{x,y\})$.
			Let $e=\{u_{i}, v_{i}\}$ be an endline edge of $G^{+}$. 
			
			\begin{figure}[!ht]
				\centering
				\begin{minipage}{\textwidth}
					\centering
					\begin{tikzpicture}[scale=0.65]
						\draw[black,] (-8,0) -- (-5,0);
						\draw[black,] (-5,0) -- (-3,-1);
						\draw[black,] (-5,0) -- (-3,-0.3);
						\draw[black,] (-5,0) -- (-3,1);
						
						\draw[black,] (5,0) -- (8,0);
						\draw[black,] (5,0) -- (3,-1);
						\draw[black,] (5,0) -- (3,-0.3);
						\draw[black,] (5,0) -- (3,1);
						
						\draw (-3,-1) node {$\bullet$};
						\draw (-2.3,-1) node {$a_{1}$};
						\draw (-3,-0.3) node {$\bullet$};
						\draw (-2.3,-0.3) node {$a_{2}$};
						\draw (-3,1) node {$\bullet$};
						\draw (-2,1) node {$a_{\Delta(G)}$};
						\draw (-5,0) node {$\bullet$};
						\draw (-5,0.5) node {$v_{i}$};
						\draw (-8,0) node {$\bullet$};
						\draw (-8,0.5) node {$u_{i}$};
						
						\draw (-2.3,0.2) node {$.$};
						\draw (-2.3,0.4) node {$.$};
						\draw (-2.3,0.6) node {$.$};
						
						\draw (0,0) node {$...$};
						
						\draw (3,-1) node {$\bullet$};
						\draw (2,-1) node {$\pi_{\alpha}(a_{1})$};
						\draw (3,-0.3) node {$\bullet$};
						\draw (2,-0.3) node {$\pi_{\alpha}(a_{2})$};
						\draw (3,1) node {$\bullet$};
						\draw (1.5,1) node {$\pi_{\alpha}(a_{\Delta(G)})$};
						\draw (5,0) node {$\bullet$};
						\draw (5,0.6) node {$\pi_{\alpha}(v_{i})$};
						\draw (8,0) node {$\bullet$};
						\draw (8,0.6) node {$\pi_{\alpha}(u_{i})$};
						
						\draw (2.3,0.2) node {$.$};
						\draw (2.3,0.4) node {$.$};
						\draw (2.3,0.6) node {$.$};
						\draw[rounded corners=15pt] (-6,-1.7) rectangle (6,1.7);
						
						\draw (6.2,1.7) node {$G$};
						
					\end{tikzpicture}
				\end{minipage}
				\caption{\em Graph $G^{+}$.}
				\label{Figure 2}
			\end{figure}
			
			Let $N_{G}(v_{i})=\{a_{1},..., a_{\Delta(G)}\}$ be the open neighborhood of $v_{i}$ in $G$. Since all vertices of $G$ have the same degree, i.e., $\Delta(G)$, we have $f(\pi_{\alpha}(e))=f(e)$ for the following reasons (see Figure \ref{Figure 2}):
			
			\begin{itemize}
				\item $f(\{v_{i}, a_{j}\})=f(\pi_{\alpha}(\{v_{i}, a_{j}\}))$ for all $1\leq j\leq \Delta(G)$,
				\item the range of $f$ is $\{1,...,\Delta(G)+1\}$,
				\item $f(\{v_{i}, u_{i}\}) \not\in \{f(\{v_{i}, a_{1}\}),...,f(\{v_{i}, a_{\Delta(G)}\})\}$, and
				\item $f(\pi_{\alpha}(\{v_{i}, u_{i}\})) \not\in \{f(\pi_{\alpha}(\{v_{i}, a_{1}\})),...,f(\pi_{\alpha}(\{v_{i}, a_{\Delta(G)}\}))\}$.
			\end{itemize}
			Thus, $\pi_{\alpha}$ is an automorphism of $G^{+}$ preserving $f$, and $\pi_{\alpha}=id_{G^{+}}$ since $f$ is a distinguishing coloring of $G^{+}$.
		\end{proof}

		By Claim \ref{Claim 2.6}, we have $\alpha=id_{G}$. Thus, $\chi_{D}'(G^{+})\geq\Delta(G)+2$.
		
		\begin{claim}\label{Claim 2.7}
			{\em $\chi_{D}'(G^{+})\leq\Delta(G)+2$.}
		\end{claim}
		
		\begin{proof}
			Fix $v_{i}\in V(G)$. Define an edge coloring $c: E(G^+) \rightarrow \{1,..., \Delta(G)+2\}$ as follows: 
			\begin{itemize}
				\item Pick any hamiltonian cycle $\{z_{1},...,z_{p}\}$ of $G$ such that $z_{1}=v_{i}$ and let 
				$c(\{z_{2i-1},z_{2i}\})=3$, $c(\{z_{2i},z_{2i+1}\})=4$, i.e., color the edges alternatively with 3 and 4.
				
				\item  Let $y_{i}$ be the vertex adjacent to $z_{i}$ that is connected by an endline edge. 
				Define, $c(\{z_{1}, y_{1}\})=1$, and $c(\{z_{j}, y_{j}\})=2$ for all $2\leq j\leq p$, i.e., we color the endline edge connected to $x$ with $1$ and color all of the remaining endline edges with $2$.
				
				\item If $G\in \{K_4, K_{3,3}\}$, then $\Delta(G)=3$. We color the remaining edges of $G$ with $5$. 
			\end{itemize}
			
			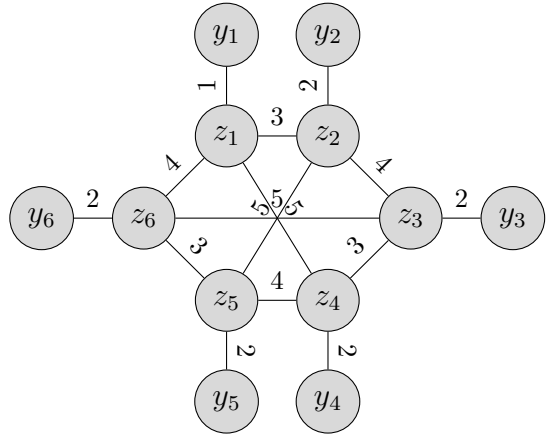
\begin{figure}[!ht]
				\centering
				\begin{minipage}{\textwidth}
					\centering
					\begin{tikzpicture}[scale=0.4,
						node distance = 5mm and 5mm,
						V/.style = {circle, draw, fill=gray!30},
						every edge quotes/.style = {auto, font=\footnotesize, sloped}
						]
						\begin{scope}[nodes=V]
							\node (1)   {$z_6$};
							\node (2) [above right=of 1]    {$z_1$};
							\node (3) [right=of 2]          {$z_2$};
							\node (4) [below right=of 3]    {$z_3$};
							\node (5) [below  left=of 4]    {$z_4$};
							\node (6) [left= of 5]          {$z_5$};
							\node (7) [right=of 4]          {$y_3$};
							\node (8) [below=of 5]          {$y_4$};
							\node (9) [below=of 6]          {$y_5$};
							\node (10) [left=of 1]  {$y_6$};
							\node (11) [above=of 2]  {$y_1$};
							\node (12) [above=of 3]  {$y_2$};
							
						\end{scope}
						\draw   (1)  edge["4"] (2)
						(2)  edge["3"] (3)
						(3)  edge["4"] (4)
						(4)  edge["3"] (5)
						(5)  edge["4"] (6)
						(6)  edge["3"] (1)
						(4)  edge["2"] (7)
						(5)  edge["2"] (8)
						(6)  edge["2"] (9)
						(1)  edge["2"] (10)
						(2)  edge["1"] (11)
						(2)  edge["5"] (5)
						(3)  edge["5"] (6)
						(1)  edge["5"] (4)
						(3)  edge["2"] (12);
					\end{tikzpicture}
				\end{minipage}
				\caption{\em 
					An edge coloring $c: E(K_{3,3}^{+}) \rightarrow \{1,..., \Delta(K_{3,3})+2\}$ of $K_{3,3}^{+}$.}
				\label{Figure 3}
			\end{figure}
			Clearly, $c$ is a proper edge coloring (see Figure \ref{Figure 3}). We show that $c$ is a distinguishing edge coloring.
			Let $\psi$ be a $c$-preserving automorphism of $G^{+}$. 
			Since $\{z_{1}, y_{1}\}$ is colored differently than any other edges in $G^{+}$, $\psi$ fixes $z_{1}$ and $y_{1}$. Moreover, $\psi$ fixes $x$ for all $x\in N_{G}(z_{1})$, since the edges $\{z_{1}, x\}$ are colored differently. 
			For $z_{p}\in V(G)\backslash \{z_{1}, z_{2}\}$, let $P_{z_{1}, z_{2}, z_{p}}$ denote the path connecting  $z_{1}, z_{2},$ and $z_{p}$.
			Now, $\psi$ fixes all the vertices of $G$ since
			for any $z_{p}, z_{q}\in V(G)\backslash \{z_{1}, z_{2}\}$ such that $z_{p}\neq z_{q}$, the length of $P_{z_{1}, z_{2}, z_{p}}$ is different than the length of $P_{z_{1}, z_{2}, z_{q}}$. By Lemma \ref{Lemma 2.4}, we have $\psi=id_{G^{+}}$.
		\end{proof}
		This concludes the proof of the present Lemma.
	\end{proof}
	
	\begin{thm}\label{Theorem 2.8}
		{\em If $G$ is a connected graph of order $n\geq 3$. Then $\chi_{D}(M(G))=\Delta(G)+1$ if $G\not\in\{C_{4}, K_{4}, C_{6}, K_{3,3}\}$. Otherwise, $\chi_{D}(M(G))=\Delta(G)+2$.}
	\end{thm}

	\begin{proof}
		By Fact \ref{Fact 2.3} (1), we have $\chi_{D}(M(G))$=
		$\chi_{D}(L(G^{+}))$. Since $G^{+}\not\in \{Q, L(Q)\}$ as $G^{+}$ has $n$ vertices of degree one (see Figure \ref{Figure 1}), we have $\chi_{D}(L(G^{+}))=\chi_{D}'(G^{+})$ by Fact \ref{Fact 2.3} (2). We show that $\chi_{D}'(G^{+})=\Delta(G)+1$ if $G\not\in\{C_{4}, K_{4}, C_{6}, K_{3,3}\}$, and $\chi_{D}'(G^{+})=\Delta(G)+2$ otherwise.  
		
		Case (i): $G\in \{C_{4}, K_{4}, C_{6}, K_{3,3}\}$. By Lemma \ref{Lemma 2.5}, we have $\chi_{D}'(G^{+})=\Delta(G)+2$.

		Case (ii): $G\not\in \{C_{4}, K_{4}, C_{6}, K_{3,3}\}$. Then, by Fact \ref{Fact 2.3} (3), either $\chi_{D}'(G)=\Delta(G)$ or $\chi_{D}'(G)=\Delta(G)+1$. Let $E$ be the set of all endline edges of $G^{+}$.
		
		Subcase (A): $\chi_{D}'(G)=\Delta(G)$. 
		Let $g:E(G)\rightarrow \{1,...,\Delta(G)\}$ be a proper distinguishing edge coloring of $G$. Define the mapping $g':E(G^{+})\rightarrow \{1,...,\Delta(G)+1\}$ as follows:
		
		\begin{center}
			$g'(e) =
			\begin{cases} 
				\Delta(G)+1 & \text{if}\, e\in E, \\
				
				g(e) & \text{if} \, e\not\in E.
			\end{cases}$
		\end{center}
		
		We show that $g'$ is a distinguishing proper coloring of $G^{+}$. Clearly, $g'$ is proper. If $\psi$ is an automorphism of $G^{+}$ preserving $g'$, then $\psi$ maps edges from $E$ to $E$. So, $\psi\restriction G$ is an automorphism of $G$. 
		Since $g$ is a distinguishing edge coloring, $\psi\restriction G=id_{G}$. 
		Thus, $\psi=id_{G^{+}}$ by Lemma \ref{Lemma 2.4}.

		Subcase (B): $\chi_{D}'(G)=\Delta(G)+1$. Let $g:E(G)\rightarrow \{1,...,\Delta(G)+1\}$ be a proper distinguishing edge coloring. For each $v\in V(G)$, let $c_{v}$ be a color from $\{1,...,\Delta(G)+1\}\backslash \{g(e):e=\{v,y\}$ for some $y\in N_{G}(v)\}$. Define $g':E(G^{+})\rightarrow \{1,...,\Delta(G)+1\}$ as follows:
		
		\begin{center}
			$g'(e) =
			\begin{cases} 
				c_{v} & \text{if}\, e\in E \text{ and } v \text{ is incident to } e, \\
				
				g(e) & \text{if} \, e\not\in E,
			\end{cases}$
		\end{center}
		
		Clearly, $g'$ is proper, and if $f$ is an automorphism of $G^{+}$ preserving $g'$, then $f$ maps edges from $E$ to $E$ since the vertices of $E$ are the only vertices with degree one in $G^{+}$.
		Similar to Subcase (A), one may verify that $g'$ is a proper distinguishing edge coloring of $G^{+}$ by applying Lemma \ref{Lemma 2.4}.
	\end{proof}
	
	\section{Distinguishing number of subdivision graphs} 
	
	\begin{defn}\label{Definition 3.1}
		Let $G = (V(G), E(G))$ be a connected graph. The {\em subdivision graph $S(G)$} of $G$ is
		a bipartite graph with the vertex set $V(G) \cup E(G)$ in which vertices $v \in V(G)$ and $e \in E(G)$ are adjacent
		if and only if the vertex $v$ is incident on the edge $e$ in the graph $G$. In other words, each edge $e = \{u, v\} \in E(G)$ is deleted and replaced by two edges $\{u, w\}$ and $\{w, v\}$ with the new
		vertex $w = \{u,v\}$.
	\end{defn}
	
	\begin{fact}\label{Fact 3.2}
		{\em Let $G$ be a connected graph
			of order $n\geq 3$. The following holds:
			\begin{enumerate}
				\item {(Kalinowski, Pil\'{s}niak, and Wo\'{z}niak; \cite[Theorem 2.2]{KPW2016})} $D''(G)\leq \lceil \sqrt{\Delta(G)}\rceil$.
				\item {(Mirafzal; \cite[Theorem 4]{Mir2024})}
				If $G \not\cong C_{n}$ then $Aut(S(G)) \cong Aut(G)$.
			\end{enumerate}
			
		}
	\end{fact}
	
	\begin{thm}\label{Theorem 3.3}
		{\em Let $G$ be a connected graph
			of order $n\geq 3$. Then $D''(G) = D(S(G))$.
		}
	\end{thm}
	
	\begin{proof}
		If $G\cong C_{n}$ where $n\geq 3$ is any natural number,
		then $D(S(G))= D(S(C_{n}))=D(C_{2n})=2=D''(C_{n})$. 
		For the rest of the proof, we assume that $G\not\cong C_{n}$ for any natural number $n$.
		
		\begin{claim}\label{Claim 3.4}
			$D''(G) \leq D(S(G))$. 
		\end{claim}
		
		\begin{proof}
			Let $f : V(S(G)) \rightarrow \{1, . . . , D(S(G))\}$ be a distinguishing vertex coloring of $S(G)$. Define, $g:V(G)\cup E(G)\rightarrow \{1, . . . , D(S(G))\}$ so that $g(x)=f(x)$ for all $x\in V(G)$ and $g(\{x,y\})=f(\{x,y\})$ for all $\{x,y\}\in E(G)=V(S(G))\backslash V(G)$. 
			We prove that $g$ is a total distinguishing coloring of $G$. Let $\alpha$ be a $g$-preserving automorphism of $G$.  
			Define the mapping $\pi_{\alpha}: V(S(G)) \rightarrow V(S(G))$ as follows:
			\begin{align*}
				\pi_{\alpha}(w) =
				\begin{cases}
					\alpha(w) & \text{if } w\in V(G)\\
					\{\alpha(x), \alpha(y)\} & \text{if } w\in V(S(G))\backslash V(G), w=\{x,y\}, x,y\in V(G).
				\end{cases}
			\end{align*}
			Then $\pi_{\alpha}\in Aut(S(G))$, where $\pi_{\alpha}(V(G)) = V(G)$ and $\pi_{\alpha}(V(S(G))\backslash V(G)) = V(S(G))\backslash V(G)$. 
			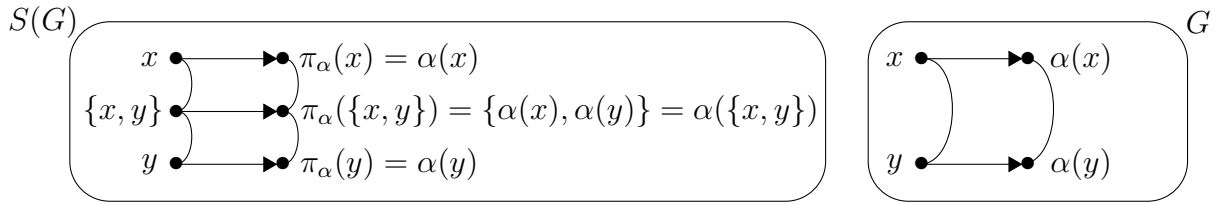
\begin{figure}[!ht]
				\centering
				\begin{minipage}{\textwidth}
					\centering
					\begin{tikzpicture}[scale=0.7]
						\draw[rounded corners=15pt] (-11,-1.7) rectangle (3.2,1.7);
						\draw[rounded corners=15pt] (4,-1.7) rectangle (10,1.7);
						\draw (10.2,1.7) node {$G$};
						\draw (5,1) node {$\bullet$};
						\draw (4.5,1) node {$x$};
						\draw (5,-1) node {$\bullet$};
						\draw (4.5,-1) node {$y$};
						\draw (7,1) node {$\bullet$};
						\draw (8,1) node {$\alpha (x)$};
						\draw (7,-1) node {$\bullet$};
						\draw (8,-1) node {$\alpha(y)$};
						\draw[black, -triangle 60] (5,1) -- (6.9,1);
						\draw[black, -triangle 60] (5,-1) -- (6.9,-1);
						\draw (5,1) to[out=5,in=5] (5,-1);
						\draw (6.9,1) to[out=5,in=5] (6.9,-1);

						\draw (-11.5,1.7) node {$S(G)$};
						\draw (-9,0) node {$\bullet$};
						\draw (-10,0) node {$\{x,y\}$};
						\draw (-9,1) node {$\bullet$};
						\draw (-9.5,1) node {$x$};
						\draw (-9,-1) node {$\bullet$};
						\draw (-9.5,-1) node {$y$};
						
						\draw (-7,0) node {$\bullet$};
						\draw (-1.8,0) node {$\pi_{\alpha}(\{x,y\})=\{\alpha(x), \alpha(y)\}=\alpha(\{x,y\})$};
						\draw (-7,1) node {$\bullet$};
						\draw (-5,1) node {$\pi_{\alpha}(x)=\alpha (x)$};
						\draw (-7,-1) node {$\bullet$};
						\draw (-5,-1) node {$\pi_{\alpha}(y)=\alpha(y)$};
						
						\draw[black, -triangle 60] (-9,0) -- (-7.1,0);
						\draw[black, -triangle 60] (-9,1) -- (-7.1,1);
						\draw[black, -triangle 60] (-9,-1) -- (-7.1,-1);
						\draw (-9,1) to[out=5,in=5] (-9,0);
						\draw (-9,0) to[out=5,in=5] (-9,-1);
						\draw (-7,1) to[out=5,in=5] (-7,0);
						\draw (-7,0) to[out=5,in=5] (-7,-1);
						
					\end{tikzpicture}
				\end{minipage}
				\caption{\em Graphs $S(G)$ and $G$.}
				\label{Figure 4}
			\end{figure}
			
			Since $\alpha$ is a $g$-preserving automorphism of $G$, we have the following:
			
			\begin{enumerate}
				\item  For all $v\in V(G)$, we have 
				$g(\alpha(v))=g(v)\implies f(\alpha(v))=f(v)\implies f(\pi_{\alpha}(v))=f(v)$. 
				
				\item  If $\{x,y\}\in V(S(G))\backslash V(G)$ and $N_{S(G)}(a)$ is the open neighborhood of $a \in V(S(G))$, then
				\begin{center}
					$\alpha(\{x,y\})\in N_{S(G)}(\alpha(x))\cap N_{S(G)}(\alpha(y))$,
				\end{center}
				since automorphisms preserve adjacency. 
				Thus, $\{\alpha(x),\alpha(y)\}=\alpha(\{x,y\})$ since $G$ is a simple graph (see Figure \ref{Figure 4}). Consequently, we have
				\begin{align*}
					g(\alpha(\{x,y\}))=g(\{x,y\}) &\implies g(\{\alpha(x),\alpha(y)\})=g(\{x,y\})\\ 
					&\implies f(\{\alpha(x),\alpha(y)\})=f(\{x,y\})\\
					&\implies f(\pi_{\alpha}(\{x,y\}))=f(\{x,y\}).
				\end{align*}
			\end{enumerate}
			
			Thus $\pi_{\alpha}$ is an automorphism of $S(G)$ preserving $f$, and so $\pi_{\alpha}=id_{S(G)}$ since $f$ is a distinguishing vertex coloring of $S(G)$. Consequently, we have $\alpha=id_{G}$.
		\end{proof}
		
		Secondly, we show that $D(S(G))\leq D''(G)$. Let $f:V(G)\cup E(G)\to \{1,...,D''(G)\}$ be a total distinguishing coloring of $G$. We define $\Tilde{f}: V(S(G)) \to \{1,...,D''(G)\}$ so that $\Tilde{f}(x)=f(x)$ for all $x\in V(G)$ and $\Tilde{f}(\{x,y\})=f(\{x,y\})$ for all $\{x,y\}\in V(S(G))\backslash V(G)=E(G)$. We claim that $\Tilde{f}$ is a distinguishing vertex coloring of the subdivision graph $S(G)$. Let $\psi$ be a $\Tilde{f}$-preserving automorphism of $S(G)$.  
		If $H = \{\pi_{\alpha}: \alpha \in Aut(G)\}$ where $\pi_{\alpha}$ is the induced automorphism of 
		$\alpha\in Aut(G)$, which is defined in the proof of Claim \ref{Claim 3.4}, then $H\cong Aut(G)$. Moreover, $H=Aut(S(G))$ since $G\not\cong C_{n}$ (cf. \cite[the proof of Theorem 4]{Mir2024}). 
		Thus, $\psi=\pi_{\alpha}$ for some $\alpha\in Aut(G)$. Furthermore, $\psi(V(G))=V(G)$ and $\psi(E(G))=E(G)$. 
		Since $\psi$ is a $\Tilde{f}$-preserving automorphism of $S(G)$, we can see that $\alpha$ is an automorphism of $G$ preserving $f$ following the arguments in the proof of Claim \ref{Claim 3.4}.\footnote{In particular, we have the following:
			
			\begin{enumerate}
				\item  for all $v\in V(G)$, 
				$\Tilde{f}(\psi(v))=\Tilde{f}(v)\implies f(\psi(v))=f(v)\implies f(\pi_{\alpha}(v))=f(v)\implies f(\alpha(v))=f(v)$. 
				\item  for all $\{x,y\}\in V(S(G))\backslash V(G)=E(G)$, 
				we have 
				$\Tilde{f}(\psi(\{x,y\}))=\Tilde{f}(\{x,y\})\implies f(\psi(\{x,y\}))=f(\{x,y\})\implies f(\pi_{\alpha}(\{x,y\}))=f(\{x,y\})\implies f(\{\alpha(x),\alpha(y)\})=f(\{x,y\}) \implies f(\alpha\{x,y\})=f(\{x,y\})$.
			\end{enumerate}
		}
		So, $\alpha=id_{G}$ since $f$ is a total distinguishing coloring of $G$. Consequently, $\psi=\pi_{\alpha}=id_{S(G)}$.
	\end{proof}
	
	\begin{corr}\label{Corollary 3.5}
		{\em If $G$ is a connected graph of order $n\geq 3$, then $D(S(G))< min\{D(G), D'(G)\}$ and $D(S(G))\leq \lceil \sqrt{\Delta(G)}\rceil$. Moreover, the bound $D(S(G))\leq \lceil \sqrt{\Delta(G)}\rceil$ is sharp.}
	\end{corr}
	
	\begin{proof}
		The first assertion follows from the fact that $D''(G)< min\{D(G), D'(G)\}$ (cf. \cite[paragraph after Proposition 1.8]{KPW2016}), Theorem \ref{Theorem 3.3}, and Fact \ref{Fact 3.2}(1). 
		To see that the bound $D(S(G))\leq \lceil \sqrt{\Delta(G)}\rceil$ is sharp, we observe that for any star $K_{1,m}$ we have
		$D(S(K_{1,m}))= \lceil \sqrt{m}\rceil$.
	\end{proof}

	\section{Distinguishing chromatic number of subdivision graphs}

	The following definition is due to Sabidussi \cite{Sab1964}.
	
	\begin{defn}\label{Definition 4.1}
		The graph $G=(V(G), E(G))$ is {\em irreducible} if for every pair of distinct vertices $x, y \in V(G)$, we have $N_{G}(x) \neq N_{G}(y)$, where $N_{G}(v)$ denotes the open neighborhood of $v\in V(G)$.
	\end{defn}

	The following results due to Mirafzal \cite{Mir2024} will be useful for our purpose.
	
	\begin{lem}{(Mirafzal; \cite[Lemma 1]{Mir2024})}\label{Lemma 4.2}
		{\em Let $G = (U \cup W, E)$, $U \cap W = \emptyset$ be a connected bipartite graph. If $f$ is an
			automorphism of the graph $G$, then $f(U) = U$ and $f(W) = W$, or $f(U) = W$ and $f(W) = U$.}
	\end{lem}
	
	\begin{lem}\label{Lemma 4.3}
		{(Mirafzal; \cite[Lemma 2]{Mir2024})}
		{\em Let $G = (U \cup W, E)$, $U \cap W = \emptyset$ be a bipartite irreducible graph. If $f$ is an
			automorphism of $G$ such that $f(u) = u$ for every $u \in U$, then $f$ is the identity automorphism
			of $G$.
		}  
	\end{lem}

	\begin{lem}\label{Lemma 4.4}
		{\em If $H$ is a connected graph where $H \not\cong C_{n}$, then for every $\varphi\in Aut(S(H))$, we have $\varphi(V(H))=V(H)$ and $\varphi(V(S(H))\setminus V(H))=V(S(H))\setminus V(H)$.
		}
	\end{lem}
	\begin{proof}
		This follows from Lemma \ref{Lemma 4.2} (cf. the proof of \cite[Theorem 4]{Mir2024}). 
	\end{proof}
	
	\subsection{Results}
	\begin{thm}\label{Theorem 4.5}
		{\em If $G$ is a connected graph where $D(G)\neq 1$ and $\chi_D(G) = 2$, then $\vert Aut (G)\vert = 2$.
		}
	\end{thm}
	
	\begin{proof}
		Let $f:V(G)\rightarrow \{1,2\}$ be a proper distinguishing vertex coloring of $G$. Then $G$ is bipartite with bipartitions $U=f^{-1}(1)$ and $W=f^{-1}(2)$. Let $\psi$ be a nontrivial automorphism of $G$. Then $\psi$ is not an $f$-preserving automorphism of $G$. 
		Thus, $f(\psi(v))\neq f(v)$ for some $v\in V(G)$. So if $v\in U$, then $\psi(v)\in W$ (and vice versa). By Lemma \ref{Lemma 4.2}, $\psi(W) = U$ and $\psi(U) = W$.
		
		\begin{claim}\label{Claim 4.6}
			{\em Let $B_{a}=\{b\in V(G): \psi(a)=b$ for some $\psi\in Aut(G)\backslash \{id_{G}\}\}$ for each $a\in V(G)$. Then $B_{a}$ is a singleton set.} 
		\end{claim}
		
		\begin{proof}
			Assume $a \in U$, $ \psi_1(a) = b_1$ and $\psi_2 (a) = b_2$ for some $\psi_1,\psi_2\in Aut(G)\backslash\{id_{G}\}$ where $b_{1}, b_{2}\in V(G)$ and $b_{1}\neq b_{2}$. Then, $b_{1}, b_{2}\in W$, and $ \psi_2\psi_1^{-1} (b_1) = b_2$, where $\psi_2\psi_1^{-1}\in Aut(G)\backslash\{id_{G}\}$. This contradicts that $\psi_2\psi_1^{-1}(W) = U$. Similarly, if $a\in W$, the proof follows.
		\end{proof}
		
		Since $D(G)\neq 1$, the graph $G$ is not asymmetric. By Claim \ref{Claim 4.6}, there can be only one nontrivial automorphism of $G$. Hence, $\vert Aut(G)\vert = 2$.
	\end{proof}

	\begin{thm}\label{Theorem 4.7}
		{\em Let $G$ be a connected graph of order $n\geq 3$. Then the following holds:
			\begin{enumerate}
				\item If $G\not\cong C_{n}$ and $D(G)\geq 3$, then $\chi_{D}(S(G))\leq D(G)$. 
				
				\item Under the assertions of (1), the bound $\chi_{D}(S(G))\leq D(G)$ is sharp.
				\item If $G\not\cong C_{n}$ and $D(G)=2$, then $\chi_{D}(S(G))=3$.
				\item If $D(G)=1$, then $\chi_{D}(S(G))=2$. 
				\item If $G\cong C_{n}$ such that $n=3$ or $n\geq 6$, then $\chi_{D}(S(G))=D(G)+1$.
				\item If $G\cong C_{n}$ such that $n=4$ or $n=5$, then $\chi_{D}(S(G))=D(G)$. 
			\end{enumerate}
		}
	\end{thm}
	\begin{proof}
		(1).  
		Let $f: V(G)\to \{1,...,D(G)\}$ be a distinguishing vertex coloring of $G$. Since $D(G)\geq 3$, we can define a vertex coloring $\Tilde{f}: V(S(G)) \to \{1,..., D(G)\}$ of $S(G)$ as follows:\footnote{In particular, we color the vertices of $V(G)$ using the coloring $f$, and we color $\{x,y\}\in V(S(G))\backslash V(G)$ with any color other than $f(x)$ and $f(y)$.} 
		$$\Tilde{f}(v)=\begin{cases}
			f(v) & \text{ if $v\in V(G)$}\\
			i & \text{ if $i\in \{1,..., D(G)\}\backslash \{f(x), f(y)\}$, $v=\{x,y\}$, $x,y\in V(G)$.}
		\end{cases}$$
		Clearly, $\Tilde{f}$ is proper. We show that $\Tilde{f}$ is a distinguishing coloring of $S(G)$. Let $\psi$ be a $\Tilde{f}$-preserving automorphism of $S(G)$. 
		Since $G\not\cong C_{n}$, we have $\psi(V(G))=V(G)$ and $\psi(V(S(G))\backslash V(G))= V(S(G))\backslash V(G)$ by Lemma \ref{Lemma 4.4}.
		Moreover, $\psi\restriction V(G)$ is an automorphism of $G$ that preserves $f$. Since $f$ is distinguishing, $\psi\restriction V(G)=id_{G}$.
		Since $S(G)$ is a bipartite irreducible graph with bipartitions $V(G)$ and $V(S(G))\backslash V(G)$, hence from Lemma \ref{Lemma 4.3}, it follows that $\psi=id_{S(G)}$. 
		
		(2). Since
		$\chi_{D}(S(K_{5}^{+}))= D(K_{5}^{+})=3$ where $K_{5}^{+}$ is the endline graph of $K_{5}$, the bound $\chi_{D}(S(G))\leq D(G)$ is sharp.
		
		(3). Let $f: V(G)\to \{1,2\}$ be a distinguishing vertex coloring of $G$. We define $\Tilde{f}: V(S(G)) \to \{1,2,3\}$ such that $$\Tilde{f}(v)=\begin{cases}
			f(v) & \text{ if $v\in V(G)$}\\
			3 & \text{ if $v=\{x,y\}$, $x,y\in V(G)$.}
		\end{cases}$$
		Clearly, $\Tilde{f}$ is proper. Following the arguments of (1), $\Tilde{f}$ is distinguishing. Thus, $\chi_D(S(G))\leq 3$. It is enough to show that $\chi_D(S(G))\neq 2$ since $\chi_{D}(S(G))\geq \chi(S(G))=2$. Let $f: V(S(G)) \to \{1,2\}$ be any proper vertex coloring. We show that $f$ cannot be a distinguishing coloring. 
		
		\begin{claim}\label{Claim 4.8}
			{\em For all $x,y\in V(G)$, $f(x) = f(y)$. Thus, $f(w) = f(v)$ for all $w,v\in V(S(G))\setminus V(G)$.}
		\end{claim}
		\begin{proof}
			Assume $f(x)\neq f(y)$ for some $x,y\in V(G)$. Since $G$ is connected, there is a path of $n$ vertices connecting $x$ and $y$ for some natural number $n$. Hence in $S(G)$, there is a path of $2n-1$ vertices connecting $x$ and $y$. But no path with an odd number of vertices can be properly colored with $2$ colors so that the end vertices get different colors.
		\end{proof}
		Since $G$ is not asymmetric, there exists a nontrivial automorphism $\alpha$ of $G$. Then $\pi_{\alpha}$, the induced automorphism of $\alpha$ which is defined in the proof of Claim \ref{Claim 3.4}, is a nontrivial automorphism of $S(G)$.
		Since $\pi_{\alpha}(V(G))=V(G)$ and $\pi_{\alpha}(V(S(G))\setminus V(G))=V(S(G))\setminus V(G)$, by applying Claim \ref{Claim 4.8}, we can see that $f(\pi_{\alpha}(v))=f(v)$ for all $v\in V(G)$ and $f(\pi_{\alpha}\{x,y\})=f(\{x,y\})$ for all $\{x,y\}\in V(S(G))\setminus V(G)$. So, $f$ cannot be a distinguishing coloring. 
		
		(4). If $D(G)=1$, then $D(S(G)) = 1$ by applying Lemma \ref{Lemma 4.3}. Hence $\chi_D(S(G)) = \chi(S(G)) = 2$, since $S(G)$ is bipartite. 
		
		(5)-(6). We note that if $G\cong C_{n}$, then $S(G)\cong C_{2n}$. The rest follows from the following: 
		\begin{itemize}
			\item If $n\geq 6$, then $\chi_{D}(S(G))=\chi_{D}(C_{2n})=3$ and $D(C_{n})=2$. 
			\item If $n=3$, then $\chi_{D}(S(G))=\chi_{D}(C_{6})=3$ and $D(C_{3})=2$. 
			\item If $n=4$ or $n=5$, then $\chi_{D}(S(G))=\chi_{D}(C_{2n})=3$ and $D(C_{n})=3$. 
		\end{itemize} 
	\end{proof}
	
	\begin{remark}
		If $G$ is a connected graph of order $n\geq 3$, then $\chi_{D}(S(G))\leq 2D''(G)$ and the bound is sharp. Collins, Hovey, and Trenk \cite[Proposition 1.1]{CHT2009} proved that $\chi_{D}(G)\leq \chi(G)D(G)$ for any graph $G$. Since $S(G)$ is a bipartite graph, we have $\chi(S(G))=2$. Consequently, we have 
		$\chi_{D}(S(G))\leq \chi(S(G))D(S(G))=2D''(G)$ by Theorem \ref{Theorem 3.3} and \cite[Proposition 1.1]{CHT2009}.
		To see that the bound is sharp, we observe that $\chi_{D}(S(C_{3}))=\chi_{D}(C_{6})=4$ and $D''(C_{3})=2$.
	\end{remark}

\end{document}